\documentclass[12pt]{amsart}

\usepackage{amsmath,amssymb,latexsym,amsthm}

\newcommand{\sumlim}{\sum\limits}

\newtheorem{thm}{Theorem}[section]

\newtheorem{lem}[thm]{Lemma}

\newtheorem{cor}[thm]{Corollary}

\newtheorem{obs}[thm]{Observation}

\newtheorem{exa}[thm]{Example}

\theoremstyle{definition}
\newtheorem{defn}[thm]{Definition}

\newtheorem{prop}[thm]{Proposition}
\newtheorem{conj}[thm]{Conjecture}
\newtheorem{clm}[thm]{Claim}
\newcommand{\een}{\end{enumerate}}
\newcommand{\blem}{\begin{lem}}
\newcommand{\elem}{\end{lem}}
\newcommand{\bcl}{\begin{clm}}
\newcommand{\ecl}{\end{clm}}
\newcommand{\bthm}{\begin{thm}}
\newcommand{\ethm}{\end{thm}}
\newcommand{\bpr}{\begin{prop}}
\newcommand{\epr}{\end{prop}}
\newcommand{\bco}{\begin{cor}}
\newcommand{\eco}{\end{cor}}
\newcommand{\bcon}{\begin{conj}}
\newcommand{\econ}{\end{conj}}
\newcommand{\bde}{\begin{defn}}
\newcommand{\ede}{\end{defn}}
\newcommand{\bex}{\begin{exa}}
\newcommand{\eexa}{\end{exa}}
\newcommand{\bobs}{\begin{obs}}
\newcommand{\eobs}{\end{obs}}
\newcommand{\bexe}{\begin{exe}}
\newcommand{\eexe}{\end{exe}}

\newcommand{\grn}{G_{r,n}}
\newcommand{\exc}{{\rm exc}}

\begin{document}

\title{On the Excedance sets of colored permutations}

\author{Eli Bagno}
\address{The Jerusalem College of Technology, Jerusalem, Israel}
\email{bagnoe@jct.ac.il}

\author{David Garber}
\address{Department of Applied Mathematics, Faculty of Sciences, Holon Institute of Technology, PO Box 305,
58102 Holon, Israel} \email{garber@hit.ac.il}

\author{Robert Shwartz}
\address {Department of Mathematics, Bar-Ilan University, 52900 Ramat-Gan, Israel}
\email{shwartr1@macs.biu.ac.il}

\begin{abstract}
We define the excedence set and the excedance word on $G_{r,n}$,
generalizing a work of Ehrenborg and Steingrimsson and use the
inclusion-exclusion principle to calculate the number of colored
permutations having a prescribed excedance word. We show some
symmetric properties as Log concavity and unimodality of a specific
sequence of excedance words.
\end{abstract}

\date{\today}
\maketitle

\section{Introduction}

Let $S_n$ be the symmetric group on $n$ letters. The parameter {\it
excedance}, which is defined on a permutation $\pi \in S_n$ by
$$\exc(\pi)=|\{i\in [n] \mid \pi(i)>i \}|,$$ is well-known.
($[n]=\{1,\dots ,n\}$, as usual).

One can also define the {\it excedance set} of a permutation $\pi$:
$${\rm Exc}(\pi)=\{i \in [n] \mid \pi(i)>i\}.$$
One can encode the excedance set of a permutation $\pi$ as a word
$w_{\pi}$ of length $n-1$ in the letters $a,b$, where $a$ in the
$i$th place means that $i \not\in {\rm Exc}(\pi)$, while $b$ in the
$i$th place means that $i \in {\rm Exc}(\pi)$. For example, the
excedance set of the permutation:
$$\pi=\left( \begin{array}{ccccc} 1&2&3&4&5 \\ 3&5&1&4&2 \end{array}\right)\in S_5$$
is encoded by the word: $bbaa$.

Let $w$ be a word in $a,b$. Ehrenborg and Steingrimsson \cite{ES}
defined:
$$[w]=\#\{\pi \in S_n \mid w_{\pi} =w \}$$
They give some recursive and explicit formulas for $[w]$. Moreover,
they consider the unimodality of the series $\{ [b^k a^{n-1-k}] \mid
0 \leq k \leq n-1 \}$ as well as some log-concavity results.

\medskip

In this paper, we generalize this idea to the colored permutation
group $G_{r,n} =\mathbb{Z}_r \wr S_n$ (defined in Sections
\ref{grn}. We consider the series
$$\{ [b^k a^{rn-1-k}] \mid 0 \leq k \leq rn-1 \},$$ show its log-concavity
and conclude its unimodality (since all elements are positive) (See
Section \ref{log concave}).
 The proof is based on finding an explicit
formula for the number $[b^k a^{rn-1-k}]$. Moreover, in Section
\ref{Inc-Exc}, we supply an algorithm to compute the number of
colored permutations having a fixed excedance, based on the
inclusion-exclusion approach presented in \cite{ES}.

\section{The group of colored permutations}\label{grn}

\bde Let $r$ and $n$ be positive integers. {\it The group of colored
permutations of $n$ digits with $r$ colors} is the wreath product:
$$\grn=\mathbb{Z}_r \wr S_n=\mathbb{Z}_r^n \rtimes S_n,$$
consisting of all the pairs $(\vec{z},\tau)$ where $\vec{z}$ is an
$n$-tuple of integers between $0$ and $r-1$ and $\tau \in S_n$. The
multiplication is defined by the following rule: for
$z=(z_1,\dots,z_n)$ and $\vec{z}'=(z'_1,\dots,z'_n)$
$$(\vec{z},\tau) \cdot (\vec{z}',\tau')=((z_1+z'_{\tau^{-1}(1)},\dots,z_n+z'_{\tau^{-1}(n)}),\tau \circ \tau')$$ (here $+$ is
taken modulo $r$). \ede

Note that the symmetric group $S_n=G_{1,n}$ and the group of
signed permutations $B_n=G_{2,n}=C_2\wr S_n$ are special cases of
$G_{r,n}$.

We use some conventions along this paper. For an element\break
$\pi=(\vec{z},\tau) \in \grn$ with $\vec{z}=(z_1,\dots,z_n)$ we
write $z_i(\pi)=z_i$. For $\pi=(\vec{z},\tau)$, we denote
$|\pi|=(\vec{0},\tau), (\vec{0} \in \mathbb{Z}_r^n)$. We also
define\break $c_i(\pi)=z_i(\pi^{-1})$ and $\vec{c}=(c_1,\dots,c_n)$.
Using this notation, the element
$(\vec{c},\tau)=\left((0,1,2,3),\begin{pmatrix} 1 & 2 & 3 &4 \\
2 & 1 & 4 &3
\end{pmatrix}\right) \in
G_{3,4}$ will be written as $(2 \bar{1} \bar{\bar{4}}
\bar{\bar{\bar{3}}}).$

Here is another way to present $\grn$: Consider the alphabet
$\Sigma=\{1,\dots,n,\bar{1},\dots,\bar{n},\dots,
1^{[r-1]},\dots,n^{[r-1]} \}$ as the set $[n]$ colored by the colors
$0,\dots,r-1$. Then, an element of $\grn$ is a {\it colored
permutation}, i.e., a bijection $\pi: \Sigma \rightarrow \Sigma$
satisfying the following condition: if
$\pi(i^{[\alpha]})=j^{[\beta]}$ then
$\pi(i^{[\alpha+1]})=j^{[\beta+1]}$.

Define the {\it color order} on $\Sigma$:

$$1^{[r-1]} <\! \cdots < n^{[r-1]} < 1^{[r-2]} < 2^{[r-2]} < \cdots < n^{[r-2]} < \cdots < 1 < \cdots\! < n.$$

%Note that any other order can be used too.

The {\it complete notation} of $\pi$ will be presented by an
example:
Let $\pi=(\vec{z},\tau)=\left((1,2,0),  \begin{pmatrix} 1 & 2 & 3 \\
3 & 1 & 2
\end{pmatrix}\right) \in G_{3,3}.$ The complete notation of
$\pi$ is:
$$\begin{pmatrix} \bar{\bar{1}} & \bar{\bar{2}} & \bar{\bar{3}} &
\bar{1} & \bar{2}& \bar{3} & 1 & 2 & 3\\
\bar{\bar{3}} & 1 & \bar{2} & \bar{3} & \bar{\bar{1}}  &  2 & 3 &
\bar{1} & \bar{\bar{2}}
\end{pmatrix}.$$
 We will use the complete notation throughout this paper.

\section{Statistics on $\grn$}\label{stat}

We start by defining the excedance set for colored permutations.

Let $\pi \in \grn$. Write $\pi$ in its complete notation and define:
$${\rm Exc}(\pi)= \{x \in \Sigma \mid \pi(x)>x\}.$$
Other definitions can be found in \cite{F,St}.

\medskip

We associate to $\pi$ the matrix $M(\pi)=(t_i^j)$ where $1 \leq i
\leq n$, $j$ varies from $r-1$ downto $0$ and $t_i^j \in \{a,b\}$ in
the following way:

$$ t_i^j= \left\{ \begin{array}{ll} a & \pi(i^{[j]}) \leq  i^{[j]}\\
                                     b &  \pi(i^{[j]})> i^{{[j]}} \end{array}\right.$$

\begin{exa}
Let  $$\pi=\begin{pmatrix} \bar{\bar{1}} & \bar{\bar{2}} &
\bar{\bar{3}} &
\bar{1} & \bar{2}& \bar{3} & 1 & 2 & 3\\
\bar{\bar{3}} & 1 & \bar{2} & \bar{3} & \bar{\bar{1}}  &  2 & 3 &
\bar{1} & \bar{\bar{2}}
\end{pmatrix}.$$
The associated matrix is:
$$M(\pi) = \begin{pmatrix} b&b &b \\b &a &b  \\b &a &a  \end{pmatrix}. $$

\end{exa} Note that $t_n^0=a$ always.

Let $\pi \in \grn$. The {\it excedance word} of $\pi$, $w_{\pi}$ is
the word obtained by reading the entries of $M(\pi)=(t_i^j)$ row by
row from left to right, ignoring the last place. Thus for $\pi$
defined above one has $w_{\pi}=bbbbabba$. We usually insert
separators between the rows of $M(\pi)$. In our example, we get
$w_{\pi}=(bbb|bab|ba)$.

For a word $w \in \{a,b\}^{rn-1}$, define $[w]$ to be the number of
elements $\pi \in \grn$ such that $w_{\pi}=w$.

\section{Unimodality and Log-concavity of the sequence $\{[b^ka^{rn-1-k}]\mid 0 \leq k \leq
rn-1\}$} \label{log concave} We recall the following definitions:

\bde A sequence of positive real numbers $a_0,a_1,\dots$ is
 {\it unimodal} if for some integers  $0 \leq k \leq n$, $t \geq 0$,  we have $$a_0
 \leq a_1 \leq \cdots \leq a_k  \geq \cdots \geq a_n.$$  \ede

\bde
 A sequence of real numbers $a_0,a_1,\dots$ is
\it{Log-concave} if for any $k>0$ one has: $a_{k-1} \cdot a_{k+1}
\leq a_k^2$. \ede

It is known that if the elements of the sequence are positive then
Log-concavity implies unimodality.

The following two observations are essential for the sequal.

\begin{obs}\label{obs 1}

 Let $\pi \in \grn$ and let $i \in \{1,\dots,n\}$. Then $c_i(\pi)=0$ if and only if all
 the entries of the $i$-th column of $M(\pi)$ are equal. In this case,  $t_i^{r-1}=a$ if and only if $\pi(i) \leq i$.

\end{obs}

\begin{obs}\label{obs 2}
Let $\pi \in \grn$ and let $i \in \{1,\dots,n\}$ be such that
$c_i(\pi)\neq 0$. Then the $i$-th column of $M(\pi)$ is of the
form: $t_i^{r-1}\cdots t_i^{0}=b^ua^{r-u}$ if and only if
$u=c_i(\pi)$.
\end{obs}

\bthm

For each $0 \leq k \leq rn-1$, one has:

$$[b^ka^{rn-1-k}]= \left\{\begin{array}{ll}
(k+1)^{n-k}k! & \quad 0 \leq k \leq n \\
& \\
n! & \quad n+1 \leq k \leq n(r-1) \\
& \\
(nr-k)^{k-nr+n}(nr-k)! & \quad n(r-1)+1 \leq k \leq rn-1
\end{array} \right.$$
\ethm

\medskip

\begin{proof}

Let $w=(b^ka^{rn-1-k})$.

We start with the case $k=0$. Then $w=a^{rn-1}$. The only element
$\pi \in \grn$ having no excedances at all is the identity element
and thus $[a^{rn-1}]=1$.

For $1 \leq k \leq n$, we have
$w=(b^ka^{n-k}|a^n|a^n|\cdots|a^n|a^{n-1})$. Let $\pi$ be such that
$w_{\pi}=w$. This means that for each $1 \leq i \leq k$ we have
$c_i(\pi)=1$ by Observation \ref{obs 2}, while for $ k< i \leq n$,
$c_i(\pi)=0$ and $\pi(i) \leq i$ by Observation \ref{obs 1}. Thus
$[w]=(k+1)^{n-k}k!$. (Indeed, first fix $\pi(i)$ for $k+1 \leq i
\leq n$ and then fix the rest $k$ places). Similar computations can
be found at \cite{BGS} and \cite{ES}.

 For $n+1 \leq k \leq n(r-1)$,
by Observation \ref{obs 2}, $c_i(\pi) \neq 0$ for each $i \in
\{1,\dots ,n\}$, and thus $[w]=n!$.

If $ n(r-1)+1 \leq k \leq rn-1$, then $w=(\underbrace{b^n | \cdots|
b^n}_{r-1}|b^{k-n(r-1)}a^{rn-1-k})$. Denote $k'=k-n(r-1)$. Then for
$1 \leq i \leq k'$ one has $c_i(\pi)=0$ and $\pi(i) >i$ and for
$k'+1 \leq i \leq n$, $c_i(\pi) \neq 0$. Hence
$[w]=(n-k')^{k'}(n-k')!$.

\end{proof}

%\begin{cor}
%The sequence $\{[b^ka^{rn-1-k}] \mid k \in \{0, \dots rn-1\}\}$ is
%weak unimodal.
%\end{cor}

%\begin{proof}
%By the symmetry property of the excedance, appearing in \cite{BGS}
%one has for each $k \in \{0,\dots, rn-1\}$,
%$[a^kb^{rn-1-k}]=[b^ka^{rn-1-k}]$, hence it is sufficient to prove
%that the sequence $[b^ka^{rn-1-k}]$ is increasing for $k \in
%\{0,\dots n\}$, but this is trivial.
%\end{proof}

\begin{cor}
The sequence $\{[b^ka^{rn-1-k}] \mid k \in \{0, \dots ,rn-1\}\}$ is
log-concave and unimodal.
\end{cor}

\begin{proof}
By the symmetry property of the excedance, appearing in \cite{BGS}
one has for each $k \in \{0,\dots, rn-1\}$,
$[a^kb^{rn-1-k}]=[b^ka^{rn-1-k}]$, hence it is sufficient to prove
that the sequence $[b^ka^{rn-1-k}]$ is increasing for $k \in
\{0,\dots n\}$, but this is trivial.

Since the sequence is positive, the sequence is unimodal too.
\end{proof}

\section{A direct computation for $[w]$ using
inclusion-exclusion}\label{Inc-Exc}

Recall that for each $\pi \in \grn$, the word $w_{\pi}$ is actually
a monomial in $a,b$. By $a+b$ in the location $i$ of $w_{\pi}$ we
mean that we do not care if $i$ is an excedance of $\pi$ or not. In
this way, we can speak about expressions $w \in \mathbb{Z}[a,b]$.

For each word $u=u_1 \cdots u_t$, define $E(u)=\{i \mid u_i=b\}$.
Let $(n_1,\dots,n_{k+1}) \in \mathbb{N}^{k+1}$ and define
$w=a^{n_1}\cdot (a+b) \cdot a^{n_2} \cdot (a+b) \cdots (a+b) \cdot
a^{n_k+1}$. Then one has \begin{small}$$\{\pi \in \grn \mid
w_{\pi}=a^{n_1}\cdot (a+b) \cdot a^{n_2} \cdot (a+b) \cdots (a+b)
\cdot a^{n_k+1}\}=\{\pi \in \grn \mid E(w_\pi) \subseteq
E(u)\}$$\end{small} where $u=a^{n_1}ba^{n_2}ba^{n_3}b \cdots
a^{n_k}ba^{n_k+1}$.

Ehrenborg and Steingrimsson \cite{ES} computed a formula for the
number of permutations of $S_n$ having a given excedance set using
the inclusion-exclusion principle. Their main results are the
following:

\blem [Ehrenborg-Steingrimsson] \label{St lemma}

For any vector $(n_1,\dots ,n_{k+1}) \in \mathbb{N}^{k+1}$ one has:

$$[a^{n_1}\cdot (a+b)\cdot a^{n_2} \cdot (a+b) \cdots (a+b) \cdot
a^{n_{k+1}}]=1^{n_1+1}\cdot 2^{n_2+1} \cdots (k+1)^{n_{k+1}+1}.$$

 \elem

\bthm [Ehrenborg-Steingrimsson] \label{St}

Let $w=a^{n_1}ba^{n_2}b \cdots a^{n_k}ba^{n_{k+1}}$ be an ab-word
with exactly $k$ b's. Then

$$[w]=\sumlim_{{\bf{r}} \in R_k}(-1)^{h({\bf r})}\cdot {\bf r}^{{\bf
n}(w)+{\bf 1}}
$$
where $$R_k=\{{\bf r}=(r_1,\dots,r_{k+1}) \mid r_1=1,r_{i+1}-r_i \in
\{0,1\}\},$$  $$h({\bf r})=|\{i \in \{1,\dots ,k \} \mid
r_i=r_{i+1}\}|,$$ ${\bf n}(w)=(n_1,\dots ,n_{k+1})$ and ${\bf
1}=(1,1,\dots ,1)$.  \ethm

 In this section we present a similar formula for excedance words
in $\grn$, based on Theorem \ref{St}.

\medskip

Define $$\Psi: M_{r,n}(\{a,b\}) \rightarrow \{a,b,a+b\}^{n-1}$$ by
$M=(t_i^j)   \mapsto w'=(w_1' \cdots w_{n-1}')$ where
$$w_i'=\left\{ \begin{array}{cc} a & \qquad t_i^j=a, \forall j \in
\{0,\dots,r-1\}\\
b & \qquad t_i^j=b, \forall j \in \{0,\dots, r-1\}\\
a+b & \qquad \mbox{Otherwise}
\end{array}\right.$$

\begin{exa}
Let
$$M= \begin{pmatrix} a&  b&b &b \\a & b &a &b  \\a & b &a &a  \end{pmatrix}
$$ which can be also written as a word: $w=(abbb|abab|aba)$.
Then $\Psi(w)=(a \quad b\quad a+b)$.
\end{exa}

Now, for each $M \in M_{r,n}(\{a,b\})$ which comes from some $\pi
\in \grn$ (or equivalently, for the associated word $w$) define
another map:
$$\varphi:\{\pi \in \grn \mid w_{\pi}=w\} \rightarrow \{\pi \in S_n
\mid w_{\pi}=\Psi(w)\}$$ by $\varphi(\pi)=|\pi|$.
 It can be easily
shown that $\varphi$ is a bijection. Thus, instead of computing
$[w]$, we compute $[\Psi(w)]$.

In order to compute $[\Psi(w)]$, following \cite{ES}, we write
each $b$ in $w'=\Psi(w)$ as $b=(a+b)-a$ and use the distributivity
of the ring $\mathbb{Z}[a,b]$ to expand $w'$ into a sum of
monomials in $a$ and $a+b$. Then we can use Lemma \ref{St lemma}
to calculate $[w']$.

We first illustrate this idea by the following example:

\begin{exa}
Let $w$ be such that
$$w'=\Psi(w)=(a \quad b \quad a+b \quad a \quad a \quad a+b \quad b \quad a).$$
Then:
\begin{eqnarray*}
[w']&=&[a \quad \mathbf{b} \quad a+b \quad a \quad a \quad a+b \quad
\mathbf{b} \quad
a]\\
& & \\
&=& [a \quad \mathbf{(a+b)-a} \quad a+b \quad a \quad a \quad a+b
\quad
\mathbf{(a+b)-a} \quad a]=\\
& & \\
&=& [a \quad \mathbf{a+b} \quad a+b \quad a \quad a \quad a+b \quad
\mathbf{a+b} \quad a]-\\
& & \\
& & - \quad [a \quad \mathbf{a+b} \quad a+b \quad a \quad a \quad
a+b \quad
\mathbf{a} \quad a]-\\
& & \\
& & - \quad [a \quad \mathbf{a} \quad a+b \quad a \quad a \quad a+b
\quad
\mathbf{a+b} \quad a]+ \\
& & \\
& & + \quad [a \quad \mathbf{a} \quad a+b \quad a \quad a \quad a+b
\quad \mathbf{a} \quad
a]=\\
& & \\
&=& [a^1 \quad a+b \quad a^0 \quad  a+b \quad a^2 \quad a+b \quad
a^0 \quad a+b \quad a^1]-\\
& & \\
& & - \quad [a^1 \quad a+b \quad a^0 \quad a+b \quad a^2 \quad a+b
\quad a^2]-\\
& & \\
& & - \quad [a^2 \quad a+b \quad a^2 \quad a+b \quad a^0 \quad
a+b \quad a^1]+ \\
& & \\
& & + \quad [a^2 \quad a+b \quad a^2 \quad a+b \quad a^2]=\\
& & \\
&=& 1^2 \cdot 2^1 \cdot 3^3 \cdot 4^1 \cdot 5^2  - 1^2 \cdot 2^1
\cdot 3^3 \cdot 4^3 - 1^3 \cdot 2^3 \cdot 3^1 \cdot 4^2 + 1^3 \cdot
2^3 \cdot 3^3=\\
& & \\
&=& 1^2 \cdot 2^1 \cdot 3^3 \cdot 4^1 \cdot 5^2  - 1^2 \cdot 2^1
\cdot 3^3 \cdot 4^1 \cdot 4^2 - \\
& & -1^2 \cdot 1^1 \cdot 2^3 \cdot 3^1 \cdot 4^2 + 1^2 \cdot 1^1
\cdot 2^3 \cdot 3^1 \cdot 3^2 = 1776
\end{eqnarray*}

\end{exa}

In general, if $\Psi(w)=(w_1' \cdots w_{n-1}')$ then define
$$I(w)=\{i \in \{1,\dots,n-1\} \mid w_i'=a+b\}$$ and let $k$ be
the number of letters in $\Psi(w)$ which are not $a$. In order to
separate between the digit $a$ and the other digits appearing in
$\Psi(w)$, we write $\Psi(w)=(a^{n_1} x_1 a^{n_2}x_2 \cdots x_k
a^{n_{k+1}})$ where $x_i \in \{a+b,b\}$.

Define: $$R_k=\left\{ {\bf r}=(r_1,\dots,r_{k+1}) \left|
\begin{array}{l} r_1=1,r_{i+1}-r_i \in \{0,1\}, \\ r_{i+1}-r_i=1 \mbox{ if } i \in
I(w) \end{array} \right. \right\},$$ and $h({\bf r})=|\{i \in
\{1,\dots ,k \} \mid r_i=r_{i+1}\}|$, ${\bf n}(w)=(n_1,\dots
,n_{k+1})$ and ${\bf 1}=(1,1,\dots ,1)$.

In the previous example, we have $k=4$ and $n(w)=(1,0,2,0,1)$.
$I(w)=\{2,3\}$, so :
$$R_k=\{(1,2,3,4,5),(1,2,3,4,4),(1,1,2,3,4),(1,1,2,3,3)\}$$

\bthm \label{our St}

Let $M \in M_{r,n}(\{a,b\})$ and let $w$ be the associated word.
Then
$$[w]=[\Psi(w)]=\sumlim_{{\bf{r}} \in R_k}(-1)^{h({\bf r})}\cdot {\bf r}^{{\bf
n}(w)+{\bf 1}} .$$
  \ethm

\begin{proof}

When we substitute $b=(a+b)-a$ in $\Psi(w)$ and expand it into a sum
of monomials in $\{a,a+b\}$, we get $2^p$ monomials where $p$ is the
number of appearances of $b$ in $\Psi(w)$. Each monomial is obtained
by substituting $a$ or $a+b$ in each instance of $b$ in $\Psi(w)$.

We can index these monomials in the following way:  Let $m=(a^{n_1}
y_1 a^{n_2} \cdots y_k a^{n_{k+1}})$ be a monomial appearing in the
expansion of $\Psi(w)$, where $y_i \in \{a,a+b\}$, $y_i=x_i$ if
$x_i=a+b$. Define $\mathbf{r}=(r_1,\dots ,r_{k+1}) \in R_k$ by:
$r_1=1$ and:
$$r_{i+1} = \left\{\begin{array}{cc}
r_i & y_i=a\\
r_i+1 & y_i=a+b \end{array}\right.$$

This gives us the $2^p$ vectors of $R_k$. Each vector can be seen as
a lattice walk of length $k$ where $r_{i+1}-r_i=1$ corresponds to a
vertical step while $r_{i+1}=r_i$ corresponds to a horizontal step.
In this view, $h(\bf{r})$ is just the number of horizontal steps
which is also the number of appearances of $a$ in $m$. The sign of
each monomial depends on the parity of $h(\mathbf{r})$ and by Lemma
\ref{St lemma} each monomial contributes
$\mathbf{r}^{\mathbf{n}(w)+\mathbf{1}}$ to $[w]$.

\end{proof}

\end{document}